\RequirePackage{ifpdf}
\ifpdf 
\documentclass[pdftex]{sigma}
\else
\documentclass{sigma}
\fi

\usepackage{graphicx,amsmath,amssymb,verbatim,amsthm}
\usepackage{color,overpic}
\usepackage{pifont}
\usepackage{tikz}
\usepackage[utf8]{inputenc}

\usetikzlibrary{positioning}
\usepackage{enumitem}
\usetikzlibrary{shapes,arrows}
\definecolor{lightgray}{gray}{0.5}

\def\frac#1#2{{\begingroup #1\endgroup\over #2}}

\thicklines

\begin{document}

\renewcommand{\PaperNumber}{***}

\FirstPageHeading

\ShortArticleName{Recurrence relations for orthogonal polynomials on a triangle}

\ArticleName{Recurrence relations for a family of orthogonal polynomials on a triangle}

\Author{Sheehan Olver~$^\ast$, Alex Townsend~$^\dag$, and Geoffrey M.~Vasil~$^\ddag$}

\AuthorNameForHeading{S. Olver, A. Townsend, and G. M. Vasil}

\Address{$^\ast$~Department of Mathematics, Imperial College, London SW7 2AZ, UK} 
\EmailD{\href{mailto:s.olver@imperial.ac.uk}{s.olver@imperial.ac.uk}} 

\Address{$^\dag$~Department of Mathematics, Cornell University, Ithaca, NY 14853, US}
\EmailD{\href{mailto:townsend@cornell.edu}{townsend@cornell.edu}} 

\Address{$^\ddag$~School of Mathematics and Statistics, University of Sydney, Australia}
\EmailD{\href{mailto:geoffrey.vasil@sydney.edu.au }{geoffrey.vasil@sydney.edu.au }} 


\ArticleDates{Received ???, in final form ????; Published online ????}

\Abstract{This paper derives sparse recurrence relations between orthogonal polynomials on a triangle and their partial derivatives, which are analogous to recurrence relations for Jacobi polynomials.  We derive these recurrences in a systematic fashion by introducing ladder operators that map an orthogonal polynomial to another by incrementing or decrementing its associated parameters by one.}


\Classification{33D50, 65Q30} 

\def\P{\smash{{P}_{n,k}^{(a,b,c,d)}(x,y)}}
\def\PP{\smash{P_{n,k}^{(a,b,c,d)}}}
\def\Pabc{\smash{{P}_{n,k}^{(a,b,c)}(x,y)}}
\def\tPabc{{\tilde P}_{n,k}^{(a,b,c)}(x,y) }
\def\PPabc{\smash{{P}_{n,k}^{(a,b,c)}}}
\def\tPPabc{{\tilde P}_{n,k}^{(a,b,c)}}
\def\ddx{\tfrac{\partial}{\partial x}}
\def\ddy{\tfrac{\partial}{\partial y}}
\def\ddz{\tfrac{\partial}{\partial z}}
\def\dudx{\tfrac{\partial u}{\partial x}}
\def\dudy{\tfrac{\partial u}{\partial y}}
\def\dudz{\tfrac{\partial u}{\partial z}}
\def\ddxy{\tfrac{\partial^2}{\partial x\partial y}}
\def\ddyy{\tfrac{\partial^2}{\partial y^2}}
\def\ddxx{\tfrac{\partial^2}{\partial x^2}}
\def\L#1{{\cal L}_{#1}}
\def\Ld#1{{\cal L}_{#1}^\dagger}
\def\tL#1{\tilde{\cal L}_{#1}}
\def\tLd#1{\tilde{\cal L}_{#1}^\dagger}
\def\M#1{{\cal M}_{#1}}
\def\Md#1{{\cal M}_{#1}^\dagger}
\def\N#1{{\cal L}_{#1}}
\def\Nd#1{{\cal L}_{#1}^\dagger}

\section{Introduction}
In 1975, Koornwinder described a general procedure for constructing multivariate orthogonal polynomials from univariate ones~\cite{Koornwinder}. The procedure allows for the construction of seven classes of bivariate orthogonal polynomials from Jacobi polynomials, some of which were previously known~\cite{Proriol_57_01}. In this paper, we consider a four-parameter variant of the {\it Koornwinder Class IV} polynomials defined as~\cite{Dunkl_14_01}
\begin{equation} 
\begin{aligned}
\P  & = P_{n-k}^{(2k+b+c+d+1,a)}\!(2x-1)(1-x)^kP_k^{(c,b)}\!\!\left(-1+\tfrac{2y}{1-x}\right)\\
& =  \tilde{P}_{n-k}^{(2k+b+c+d+1,a)}\!(x)(1-x)^k\tilde{P}_k^{(c,b)}\!\!\left(\tfrac{y}{1-x}\right),
\end{aligned}
\label{eq:Koornwinder}
\end{equation} 
where $a,b,c>-1$, $n$ and $k$ are integers such that $n\geq k\geq 0$, $\smash{P_k^{(a,b)}(x)}$ is the Jacobi polynomial of degree $k$~\cite[Table 18.3.1]{DLMF}, and $\smash{\tilde{P}_k^{(a,b)}}$ is the Jacobi polynomial of degree $k$ shifted to have support on $(0,1)$. The Koornwinder Class IV polynomials are a special case when $d = 0$, which we denote by $\PPabc$. The polynomials in~\eqref{eq:Koornwinder} are orthogonal on the right-angled triangle $\{(x,y) : 0 < x < 1, 0 < y < 1-x\}$ with respect to the weight function $w_{a,b,c,d}(x,y) = x^a y^b (1-x-y)^c(1-x)^d$. That is, if $n\neq m$ or $k\neq \ell$ then
$$
\int_{0}^1\int_{0}^{1-x} w_{a,b,c,d}(x,y) \P P_{m,\ell}^{(a,b,c,d)}(x,y) dy dx= 0. 
$$

The introduction of the fourth parameter in~\eqref{eq:Koornwinder} allows us to schematically derive sparse recurrence relations that are satisfied by $\PPabc$ (see Section~\ref{sec:recurrences}).  To do this, we  derive ladder operators that map $\smash{\P}$ to a scalar multiple of $\smash{P_{\tilde n,\tilde k}^{(\tilde a,\tilde b,\tilde c,\tilde d)}}$, where the new parameters in $\smash{P_{\tilde n,\tilde k}^{(\tilde a,\tilde b,\tilde c,\tilde d)}}$ are $n$, $k$, $a$, $b$, $c$ or $d$, respectively, incremented or decremented by $0$ or $1$ (see Section~\ref{sec:ladder}). 

The recurrence relations can be employed to efficiently solve linear partial differential equations defined on a triangle using sparse linear algebra, analogous to the ultraspherical spectral method for solving ordinary differential equations on bounded intervals~\cite{SOATUltraspherical}. A similar idea using a hierarchy of Zernike polynomials, which are bivariate orthogonal polynomials on the unit disk, is used in~\cite{GVDisk} to develop a sparse spectral method for solving partial differential equations defined on the disk~\cite{GVDisk}. On the disk, polar coordinates allow for radially symmetric partial differential operators to be reduced to ordinary differential operators acting on Jacobi polynomials~\cite{GVDisk}. This simplification does not translate to non-radially symmetric partial differential operators on the disk, nor partial differential operators on the triangle.  

Several of the formulae in this paper have already be derived by directly employing recurrence relations satisfied by Jacobi polynomials~\cite{Xu_TA}. Our approach via ladder operators is a more systematic study that derives previously unreported recurrence relations for $\PPabc$. We also hope to use ladder operators to derive sparse recurrence relations for multivariate orthogonal polynomials built from Jacobi polynomials on higher-dimensional simplices.

Throughout this paper, the recurrence relations hold for choices of the parameters $n$, $k$, $a$, $b$, $c$, and $d$ that make the Jacobi polynomials well-defined. Moreover, we take $\smash{P_{-1}^{(a,b)}(x) = 0}$, which gives
$$
0 = P_{-1,0}^{(a,b,c,d)}(x,y) = P_{n,-1}^{(a,b,c,d)}(x,y)  = P_{n,n+1}^{(a,b,c,d)}(x,y), \qquad n\geq -1. 
$$
Also, note that orthogonal polynomials remain orthogonal after an affine transformation so the recurrence relations in this paper for~\eqref{eq:Koornwinder} on a right-angled triangle can be extended to any triangle, including   triangles with the corners permuted.

The paper is structured as follows. In the next section, we give 12 ladder operators for Jacobi polynomials and use them to derive sparse recurrence relations for $\smash{P_n^{(a,b)}}$. In Section~\ref{sec:ladder} we give 24 ladder operators for~\eqref{eq:Koornwinder} and write down the corresponding sparse recurrence relations for $\PP$. 
In Section~\ref{sec:recurrences}, we use the ladder operators to derive a collection of sparse recurrence relations for differentiation, conversion, and multiplication that are satisfied by $\PPabc$. 

\section{Ladder operators for Jacobi polynomials}\label{sec:JacobiLadder}
We give 12 ordinary differential operators that increment or decrement the parameters and degree of Jacobi polynomials by zero or one. Each ladder operator maps $\smash{P_n^{(a,b)}}(x)$ to $\smash{P_{\tilde{n}}^{(\tilde{a},\tilde{b})}}(x)$, where $|\tilde{n} - n|\leq 1$, $|\tilde{a} - a|\leq 1$, and $|\tilde{b} - b|\leq 1$. 
\begin{definition}
The following operators are ladder operators for Jacobi polynomials:
$$
\begin{aligned}
\L1u &= \tfrac{du}{dx} & \Ld1u & =  ((1+x)a - (1-x) b)u - (1-x^2) \tfrac{du}{dx} \\
\L2u &= (a+b+n+1)u + (1+x)\tfrac{du}{dx} & \Ld2u & = (2a + (1-x) n)u-(1-x^2)\tfrac{du}{dx}  \\
\L3u &= (a+b+n+1)u-(1-x)\tfrac{du}{dx} &\Ld3u &= (2b+(1+x)n)u+(1-x^2)\tfrac{du}{dx} \\
\L4u &= ((1+x)a-(1-x)(b+n+1))u-(1-x^2)\tfrac{du}{dx} &\Ld4u &=-nu+(1+x)\tfrac{du}{dx} \\
\L5u& = ((1+x)(a+n+1)-(1-x)b)u-(1-x^2)\tfrac{du}{dx} &\Ld5u&=nu+(1-x)\tfrac{du}{dx} \\
\L6u&=bu+(1+x)\tfrac{du}{dx} & \Ld6u&=au-(1-x)\tfrac{du}{dx}.
\end{aligned}
$$
\label{def:JacobiLadder}
\end{definition}

The notation for the ladder operators is chosen so that $\Ld s\L s P_n^{(a,b)}$ and $\L s\Ld s P_n^{(a,b)}$ are scalar multiples of $\smash{P_n^{(a,b)}}$ for $1\leq s\leq 6$. These ladder operators are carefully constructed to give rise to sparse recurrence relations for Jacobi polynomials. 

\begin{lemma}
The ladder operators give sparse recurrence relations for Jacobi polynomials: 
$$
\begin{aligned}
\L1 P_n^{(a,b)} &= \tfrac{1}{2}(n+a+b+1) P_{n-1}^{(a+1,b+1)} & \Ld1 P_n^{(a,b)} &= 2(n+1) P_{n+1}^{(a-1,b-1)} \\
\L2 P_n^{(a,b)} &= (n+a+b+1) P_{n}^{(a+1,b)} & \Ld2 P_n^{(a,b)} &= 2(n+a) P_{n}^{(a-1,b)}\\
\L3 P_n^{(a,b)} &= (n+a+b+1) P_{n}^{(a,b+1)} & \Ld3 P_n^{(a,b)} &= 2(n+b) P_{n}^{(a,b-1)}\\
\L4 P_n^{(a,b)} &= 2(n+1) P_{n+1}^{(a-1,b)} & \Ld4 P_n^{(a,b)}&= (n+b) P_{n-1}^{(a+1,b)}\\
\L5 P_n^{(a,b)} &= 2(n+1) P_{n+1}^{(a,b-1)} & \Ld5 P_n^{(a,b)} &= (n+a) P_{n-1}^{(a,b+1)}\\
\L6 P_n^{(a,b)} &= (n+b) P_{n}^{(a+1,b-1)} & \Ld6 P_n^{(a,b)}& = (n+a) P_{n}^{(a-1,b+1)}.
\end{aligned}
$$
\label{lem:JacobiRecurrences}
\end{lemma}	
\begin{proof}
The recurrence relations can be easily verified by showing that the left-hand side has the same orthogonality properties and normalization constant as the right-hand side.
\end{proof} 

The sparse recurrence relations corresponding to $\L1$ and $\Ld1$ are already known: one is a formula for the derivative of $\smash{P_n^{(a,b)}(x)}$~\cite[18.9.15]{DLMF} and the other is equivalent to~\cite[18.9.16]{DLMF}. We have not found the other 10 recurrence relations in the literature.  Figure~\ref{fig:LadderDiagram} illustrates the ladder operators and how they increment or decrement the parameters associated to a Jacobi polynomial. 
\begin{figure} 
\centering 
\begin{minipage}{.32\textwidth} 
\centering
\begin{tikzpicture} 
\draw[thick,black,->] (1,1) -- (0,1);
\node(a1) at (.5,1) {{\footnotesize{$\L4$}}};
\draw[thick,black,->] (1,1) -- (0,0);
\node(a1) at (.5,.5) {{\footnotesize{$\Ld1$}}};
\draw[thick,black,->] (1,1) -- (1,0);
\node(a1) at (1,.5) {{\footnotesize{$\L5$}}};
\node(a) at (2,2) {};
\draw[dashed,-] (0,0) -- (2,0) -- (2,2) -- (0,2) -- (0,0);
\draw[dashed,-] (1,0) -- (1,2);
\draw[dashed,-] (0,1) -- (2,1);
\node(b) at (1,-.3) {\footnotesize{$a$}};
\node(b) at (2,-.3) {\footnotesize{$a+1$}};
\node(b) at (0,-.3) {\footnotesize{$a-1$}};
\node(b) at (-.5,1) {\footnotesize{$b$}};
\node(b) at (-.5,2) {\footnotesize{$b+1$}};
\node(b) at (-.5,0) {\footnotesize{$b-1$}};
\node(b) at (1,2.5) {\footnotesize{$P_n^{(a,b)}\rightarrow P_{n-1}^{(\tilde{a},\tilde{b})}$}};
\end{tikzpicture} 
\end{minipage} 
\begin{minipage}{.32\textwidth} 
\centering
\begin{tikzpicture} 
\draw[thick,black,->] (1,1) -- (2,0);
\node(a1) at (1.5,.5) {{\footnotesize{$\L6$}}};
\draw[thick,black,->] (1,1) -- (0,2);
\node(a1) at (.5,1.5) {{\footnotesize{$\Ld6$}}};
\draw[thick,black,->] (1,1) -- (0,1);
\node(a1) at (.5,1) {{\footnotesize{$\Ld2$}}};
\draw[thick,black,->] (1,1) -- (2,1);
\node(a1) at (1.5,1) {{\footnotesize{$\L2$}}};
\draw[thick,black,->] (1,1) -- (1,2);
\node(a1) at (1,1.5) {{\footnotesize{$\L3$}}};
\draw[thick,black,->] (1,1) -- (1,0);
\node(a1) at (1,.5) {{\footnotesize{$\Ld3$}}};
\draw[dashed,-] (0,0) -- (2,0) -- (2,2) -- (0,2) -- (0,0);
\draw[dashed,-] (1,0) -- (1,2);
\draw[dashed,-] (0,1) -- (2,1);
\node(b) at (1,-.3) {\footnotesize{$a$}};
\node(b) at (2,-.3) {\footnotesize{$a+1$}};
\node(b) at (0,-.3) {\footnotesize{$a-1$}};
\node(b) at (-.5,1) {\footnotesize{$b$}};
\node(b) at (-.5,2) {\footnotesize{$b+1$}};
\node(b) at (-.5,0) {\footnotesize{$b-1$}};
\node(b) at (1,2.5) {\footnotesize{$P_n^{(a,b)}\rightarrow P_n^{(\tilde{a},\tilde{b})}$}};
\end{tikzpicture} 
\end{minipage} 
\begin{minipage}{.32\textwidth} 
\centering
\begin{tikzpicture} 
\draw[thick,black,->] (1,1) -- (2,1);
\node(a1) at (1.5,1) {{\footnotesize{$\Ld4$}}};
\draw[thick,black,->] (1,1) -- (2,2);
\node(a1) at (1.5,1.5) {{\footnotesize{$\L1$}}};
\draw[thick,black,->] (1,1) -- (1,2);
\node(a1) at (1,1.5) {{\footnotesize{$\Ld5$}}};
\node(a) at (0,0) {};
\draw[dashed,-] (0,0) -- (2,0) -- (2,2) -- (0,2) -- (0,0);
\draw[dashed,-] (1,0) -- (1,2);
\draw[dashed,-] (0,1) -- (2,1);
\node(b) at (1,-.3) {\footnotesize{$a$}};
\node(b) at (2,-.3) {\footnotesize{$a+1$}};
\node(b) at (0,-.3) {\footnotesize{$a-1$}};
\node(b) at (-.5,1) {\footnotesize{$b$}};
\node(b) at (-.5,2) {\footnotesize{$b+1$}};
\node(b) at (-.5,0) {\footnotesize{$b-1$}};
\node(b) at (1,2.5) {\footnotesize{$P_n^{(a,b)}\rightarrow P_{n+1}^{(\tilde{a},\tilde{b})}$}};
\end{tikzpicture} 
\end{minipage} 
\caption{Illustration of the 12 ladder operators for Jacobi polynomials in Definition~\ref{def:JacobiLadder}. }
\label{fig:LadderDiagram}
\end{figure}
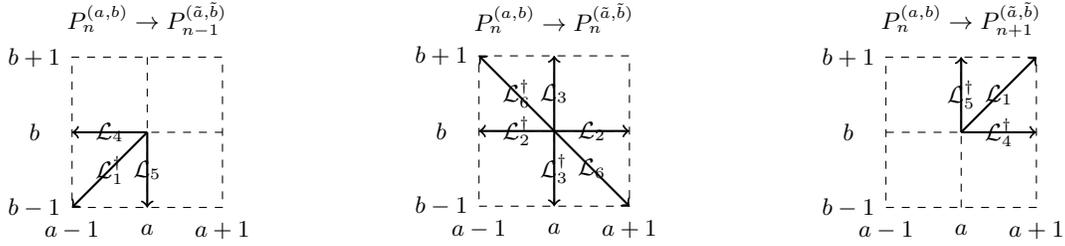 

The ladder operators can be easily adapted to the shifted Jacobi polynomials, denoted by $\tilde{P}_k^{(a,b)}$, which are supported on $(0,1)$. 
\begin{definition}
The following operators are ladder operators for shifted Jacobi polynomials:
$$
\begin{aligned}
\tL1u &= \tfrac{du}{dx} & \tLd1u & =  (xa - (1-x) b)u - x(1-x) \tfrac{du}{dx} \\
\tL2u &= (a+b+n+1)u + (1+x)\tfrac{du}{dx} & \tLd2u & = (a + (1-x) n)u-x(1-x)\tfrac{du}{dx}  \\
\tL3u &= (a+b+n+1)u-(1-x)\tfrac{du}{dx} &\tLd3u &= (b+(1+x)n)u+x(1-x)\tfrac{du}{dx} \\
\tL4u &= (xa-(1-x)(b+n+1))u-x(1-x)\tfrac{du}{dx} &\tLd4u &=-nu+(1+x)\tfrac{du}{dx} \\
\tL5u &= (x(a+n+1)-(1-x)b)u-x(1-x)\tfrac{du}{dx}  &\tLd5u&=nu+(1-x)\tfrac{du}{dx} \\
\tL6u&=bu+x\tfrac{du}{dx} &\tLd6u&=au-(1-x)\tfrac{du}{dx}.
\end{aligned} 
$$
\label{def:JacobiLadder2}
\end{definition}
The corresponding recurrence relations for $\tilde{P}_k^{(a,b)}$ are the same as in Lemma~\ref{lem:JacobiRecurrences}, except the multiplicative factors of $\tfrac{1}{2}$ and $2$ are replaced by $1$. 

\section{Ladder operators for $\PP$}\label{sec:ladder}
The 12 ladder operators for the Jacobi polynomials in Section~\ref{sec:JacobiLadder} allow us to derive 24 ladder operators for $\PP$. The ladder operators are carefully defined so that they map $\smash{\PP}$ to a scalar multiple of $\smash{P_{\tilde n,\tilde k}^{(\tilde a,\tilde b,\tilde c,\tilde d)}}$, where the new parameters in $\smash{P_{\tilde n,\tilde k}^{(\tilde a,\tilde b,\tilde c,\tilde d)}}$ are $n$, $k$, $a$, $b$, $c$ or $d$, respectively, incremented or decremented by $0$ or $1$.

To highlight the symmetries of the right-angled triangle and make the recurrences more convenient to write down, we define the variable  $z = 1-x-y$ and  
$$
\ddz = \ddy - \ddx,
$$
as in \cite{Xu_TA}. Now, the variables $x$, $y$, and $z$ have the convenient property that any affine transformation that maps the triangle onto itself has the effect of exchanging the roles of $x$, $y$, and $z$.

\begin{definition} 
The following operators are ladder operators for $\PP$. The first set of 12 are:
$$
\begin{aligned}
\M{0,1}u& = \dudy &   \Md{0,1}u &= (y c - z  b)u - y z\dudy \\
\M{0,2}u &=(k+b+c+1)u+y\dudy & \Md{0,2}u &= \left(c+k-\tfrac{y k}{1-x}\right)u - \tfrac{y}{1-x} z \dudy \\
\M{0,3}u &=(k+b+c+1)u-x\dudy & \Md{0,3}u&=\left(b+\tfrac{k y}{1-x}\right)u + \tfrac{y}{1-x} z \dudy \\
\M{0,4}u &= (y c - z (b + k + 1))u  - yz \dudy & \Md{0,4}u &=-\tfrac{k}{1-x}u + \tfrac{y}{1-x} \dudy \\
\M{0,5}u&= (y (c + k + 1))u -  z b - yz \dudy & \Md{0,5}u&= \tfrac{k}{1-x}u + \left(1-\tfrac{y}{1-x}\right) \dudy \\
\M{0,6}u & = cu-z \dudy & \Md{0,6}u& = bu+y\dudy.\\
\end{aligned}
$$
The second set of 12 are:
$$
\begin{aligned}
\M{1,0}u &= \tfrac{k}{1-x}u + \dudx -\tfrac{y}{1-x} \dudy \\
\Md{1,0}u &= (x (k + a + b + c + d + 1) - a)u  - x (1 - x) \dudx+ x y \dudy \\
\M{2,0}u &= (n+ k + a + b + c + d + 2)u + \tfrac{x k}{1-x}u + x \dudx - \tfrac{x y}{1 - x} \dudy \\ 
\Md{2,0}u & = (n + k + b + c + d + 1 - x n)u - x (1 - x) \dudx+ x y \dudy \\
\M{3,0}u &= (n+a+b+c+d+2)u-(1-x) \dudx + y \dudy \\
\Md{3,0}u &= (a+xn)u+x(1-x)\dudx -x y \dudy \\
\M{4,0}u &= (x (n + a + b + c + d + 2) - a - n + k - 1)u - x (1 - x) \dudx+  x y \dudy \\
\Md{4,0}u &=\tfrac{k}{1-x}u - nu  + x \dudx - \tfrac{x y}{1-x}\dudy\\
\M{5,0}u &= nu+(1-x) \dudx- y \dudy \\
\Md{5,0}u &= x(n+a+b+c+d+2)u-au-x(1-x) \dudx +x y \dudy \\
\M{6,0}u &= au+\tfrac{x k}{1-x}u + x \dudx - \tfrac{xy}{1-x} \dudy \\
\Md{6,0}u &= (k+b+c+d+1)u-(1-x)\dudx + y \dudy. \\
\end{aligned}
$$
\label{def:Pladder}
\end{definition} 

The notation for the ladder operators is chosen so that the recurrence relations in Theorem~\ref{thm:ladder} are derived for $\M{s,0}$ (resp.~$\M{0,s}$) by applying $\L s$ or $\Ld s$ to the first (resp.~second) Jacobi polynomial in $\P$ for $1\leq s\leq 6$.  Moreover, we know that $\smash{\Md{s,0}\M{s,0} \PP}$, $\smash{\Md{0,s}\M{0,s} \PP}$, $\smash{\M{s,0}\Md{s,0} \PP}$, and $\smash{\M{0,s}\Md{0,s} \PP}$ are scalar multiples of $\smash{\PP}$ for $1\leq s\leq 6$. 

\begin{theorem}\label{thm:ladder}
The ladder operators in Definition~\ref{def:Pladder} correspond to 24 sparse recurrence relations for $\PP$. Let $t = a+b+c+d$. The first set of 12 are:
$$
\begin{aligned}
\M{0,1} \PP &= (k+b+c+1) {P}_{n-1,k-1}^{(a,b+1,c+1,d)} & \Md{0,1}\PP = (k+1) {P}_{n+1,k+1}^{(a,b-1,c-1,d)} \\
\M{0,2}\PP &= (k+b+c+1) {P}_{n,k}^{(a,b,c+1,d-1)} & \Md{0,2}\PP = (k+c) {P}_{n,k}^{(a,b,c-1,d+1)} \\
\M{0,3}\PP &= (k+b+c+1) {P}_{n,k}^{(a,b+1,c,d-1)} & \Md{0,3}\PP = (k+b) {P}_{n,k}^{(a,b-1,c,d+1)}   \\
\M{0,4}\PP &= (k+1) {P}_{n+1,k+1}^{(a,b,c-1,d-1)} & \Md{0,4}\PP = (k+b) {P}_{n-1,k-1}^{(a,b,c+1,d+1)}  \\
\M{0,5}\PP &= (k+1) {P}_{n+1,k+1}^{(a,b-1,c,d-1)} & \Md{0,5}\PP = (k+c) {P}_{n-1,k-1}^{(a,b+1,c,d+1)} \\ 
\M{0,6} \PP &= (k+c) {P}_{n,k}^{(a,b+1,c-1,d)} & \Md{0,6} \PP = (k+b) {P}_{n,k}^{(a,b-1,c+1,d)}.
\end{aligned}
$$
The second set of 12 are: 
$$
\begin{aligned}
\M{1,0}\PP &= (n+k+t+2) {P}_{n-1,k}^{(a+1,b,c,d+1)} & \Md{1,0}\PP &= (n-k+1) {P}_{n+1,k}^{(a-1,b,c,d-1)} \\
\M{2,0}\PP &= (n+k+t+2) {P}_{n,k}^{(a,b,c,d+1)} & \Md{2,0}\PP &=(n+k+t-a+1) {P}_{n,k}^{(a,b,c,d-1)}\\
\M{3,0} \PP &= (n+k+t+2) {P}_{n,k}^{(a+1,b,c,d)} & \Md{3,0} \PP &= (n-k+a) {P}_{n,k}^{(a-1,b,c,d)} \\
\M{4,0}\PP &= (n-k+1) {P}_{n+1,k}^{(a,b,c,d-1)} & \Md{4,0}\PP &= (n-k+a) {P}_{n-1,k}^{(a,b,c,d+1)} \\
\M{5,0}\PP &= (n+k+t-a+1) {P}_{n-1,k}^{(a+1,b,c,d)} & \Md{5,0} \PP &= (n-k+1) {P}_{n+1,k}^{(a-1,b,c,d)}  \\
\M{6,0}\PP &= (n-k+a) {P}_{n,k}^{(a-1,b,c,d+1)} & \Md{6,0}\PP &= (n+k+t-a+1) {P}_{n,k}^{(a+1,b,c,d-1)}.
\end{aligned}
$$
\end{theorem}
\begin{proof}
We present the proof of $\smash{\M{0,1} \PP = (k+b+c+1) {P}_{n-1,k-1}^{(a,b+1,c+1,d)}}$. By the definition of $\PP$ in~\eqref{eq:Koornwinder}, the chain rule, and the relationship in~\eqref{JJp}, we have
\begin{equation}
\begin{aligned} 
\ddy P_{n,k}^{(a,b,c,d)}(x,y) &= \tilde P_{n-k}^{(2k+b+c+d+1,a)}(x)(1-x)^{k-1}  [\tilde P_k^{(c,b)}]'\left(\tfrac{y}{1-x}\right)\\
& = (k+c+b+1)  \tilde P_{n-k}^{(2k+b+c+d+1,a)}(x)(1-x)^{k-1} \tilde P_{k-1}^{(c+1,b+1)}\left(\tfrac{y}{1-x}\right),
\end{aligned} 
\label{eq:derivation1} 
\end{equation} 
where the last equality comes from applying $\tL1$ in Definition~\ref{def:JacobiLadder2}. The final expression in~\eqref{eq:derivation1} is equivalent to $\smash{(k+b+c+1)P_{n-1,k-1}^{(a,b+1,c+1,d)}}$. The manipulations for the remaining recurrence relations are similar, except with different choices of the operators $\tL s$ or $\tLd s$ and combinations of~\eqref{JJp} and~\eqref{JpJ}.  
\end{proof}

\section{Sparse recurrence relations for $\PPabc$}\label{sec:recurrences}
We can combine the ladder operators in Section~\ref{sec:ladder} to derive sparse recurrence relations between Koornwinder polynomials with different parameters and their partial derivatives. These recurrence relations are analogous to many of the sparse recurrence relations for Jacobi polynomials~\cite[\S18.9]{DLMF}. 

\subsection{Differentiation}
The partial derivatives of $\PPabc$ can be written in terms of Koornwinder polynomials with incremented parameters, which is analogous to a recurrence relation for the derivative of a Jacobi polynomial~\cite[18.9.15]{DLMF}. A similar recurrence for $\PPabc$ can be found in~\cite[Prop.~4.6, 4.7, \& 4.8]{Xu_TA}. 
\begin{corollary}
The following recurrence relations hold:
\begin{align}
(2k+b+c+1) \ddx \PPabc &= (n+k+a+b+c+2)(k+b+c+1){P}_{n-1,k}^{(a+1,b,c+1)} \notag\\
&\qquad  \qquad + (k+b)(n+k+b+c+1){P}_{n-1,k-1}^{(a+1,b,c+1)}, \label{ddx} \\
\ddy \PPabc  &= (k + b + c + 1) {P}_{n-1,k-1}^{(a,b+1,c+1)},  \label{ddy}\\
(2k+b+c+1)\ddz\PPabc  & = -(n+k+a+b+c+2) (k+b+c+1) {P}_{n-1,k}^{(a+1,b+1,c)} \notag \\
   & \qquad \qquad + (k+c)(n+k+b+c+1) {P}_{n-1,k-1}^{(a+1,b+1,c)}.    \label{ddz}
\end{align}
\end{corollary}
\begin{proof}
The recurrence~\eqref{ddx} follows from the fact that $(\M{1,0}\M{0,2} + \Md{0,4}\Md{6,0})u = (2k+b+c+1) \dudx$ when $d = 0$. The relationship~\eqref{ddy} is equivalent to the relation given by $\M{0,1}$ in Theorem~\ref{thm:ladder} when $d = 0$. Finally,~\eqref{ddz} follows from the fact that $(\M{1,0}\M{0,3} - \Md{0,5}\Md{6,0})u = -(2k+b+c+1) \dudz$.
\end{proof}

The derivatives of weighted versions of $\PPabc$ also satisfy sparse recurrence relations, which are analogous to an expression for the derivative of a weighted Jacobi polynomial~\cite[18.9.16]{DLMF}.
\begin{corollary}
The following recurrence relations hold:
$$
\begin{aligned}
	-(2k+b+c+1)\ddx\!\left(x^a y^b z^c \PPabc\right) &= x^{a-1}y^b z^{c-1}\Big( (k+c)(n-k+1){P}_{n+1,k}^{(a-1,b,c-1)} \\
	&\qquad\qquad\qquad + (k+1)(n-k+a){P}_{n+1,k+1}^{(a-1,b,c-1)}\Big), \\
	\ddy\!\left(x^a y^b z^c \PPabc\right) &= -(k+1)x^ay^{b-1} z^{c-1} {P}_{n+1,k+1}^{(a,b-1,c-1)}, \\
	(2k+b+c+1) \ddz\! \left(x^a y^b z^c \PPabc\right) &=  x^{a-1}y^{b-1} z^c \Big((k+b)(n-k+1){P}_{n+1,k}^{(a-1,b-1,c)} \\
	&\qquad\qquad\qquad - (k+1)(n-k+a){P}_{n+1,k+1}^{(a-1,b-1,c)}\Big).
\end{aligned} 
$$
\end{corollary}
\begin{proof}
	The first recurrence follows from
$$
\begin{aligned}
	(\Md{0,2}\Md{1,0} + \M{0,4}\M{6,0})u = (2k+b+c+1)(cx -a z -xz \ddx)u\\
	= - (2k+b+c+1) x^{1-a}  z^{1-c} \ddx(x^a  z^c  u).
\end{aligned} 
$$
The second recurrence holds since 
$$
\Md{0,1}u = (cy -b z-y z \ddy)u = -y^{1-b} z^{1-c} \ddx(y^b  z^c u).
$$
The third recurrence is derived from the fact that 
$$
\begin{aligned}
	(\Md{0,3}\Md{1,0} - \M{0,5}\M{6,0})u = (2k+b+c+1)(bx - a y + x y (\ddy - \ddx))u \\
	= (2k+b+c+1) x^{1-a} y^{1-b} z^{-c} \ddz(x^a y^b z^c u).
\end{aligned} 
$$
\end{proof}

\subsection{Conversion}
Recurrence relations for conversion allow us to express $\PPabc$ in terms of Koornwinder polynomials with different parameters. Here, we give the recurrence relations that increment the parameters, which are analogues of~\cite[18.9.3]{DLMF}.  Similar relations can be found in~\cite[Prop.~4.4]{Xu_TA}.
\begin{corollary}
The following recurrence relations hold:
\begin{align}
(2n+a+b+c+2)\PPabc &= (n+k+a+b+c+2){P}_{n,k}^{(a+1,b,c)} \cr
&+ (n+k+b+c+1){P}_{n-1,k}^{(a+1,b,c)},\label{Ca}\\
(2n+a+b+c+2)(2k+b+c+1)\PPabc &= (n+k+a+b+c+2)(k+b+c+1){P}_{n,k}^{(a,b+1,c)} \notag\\
&- (n-k+a)(k+b+c+1){P}_{n-1,k}^{(a,b+1,c)}\notag\\
&+(k+c)(n+k+b+c+1){P}_{n-1,k-1}^{(a,b+1,c)}\notag\\
&-(k+c)(n-k+1){P}_{n,k-1}^{(a,b+1,c)},\label{Cb}\\
(2n+a+b+c+2)(2k+b+c+1)\PPabc &= (n+k+a+b+c+2)(k+b+c+1){P}_{n,k}^{(a,b,c+1)} \notag\\
&- (n-k+a)(k+b+c+1){P}_{n-1,k}^{(a,b,c+1)}\notag\\
&-(k+b)(n+k+b+c+1){P}_{n-1,k-1}^{(a,b,c+1)}\notag\\
&+(k+b)(n-k+1){P}_{n,k-1}^{(a,b,c+1)}.\label{Cc}
\end{align}
\label{cor:conversion}
\end{corollary}
\begin{proof}
The recurrence relation in~\eqref{Ca} follows from the fact that $(\M{30} + \M{50})u = (2n+a+b+c+2)u$ when $d = 0$. Since $(\M{2,0} - \Md{4,0})u = (2n+a+b+c+d+2)u$ and $(\Md{2,0}-\M{4,0})u = (2n+a+b+c+d+2)(1-x)u$, we obtain
$$
(\M{0,2}\M{2,0} - \M{0,2}\Md{4,0} - \Md{0,4}\Md{2,0} +\Md{0,4}\M{4,0})u = (2k+b+c+1)(2n+a+b+c+d+2)u,
$$
The recurrence relation~\eqref{Cc} immediately follows. Similarly,~\eqref{Cb} holds since
$$
(\M{0,3}\M{2,0} - \M{0,3}\Md{4,0} +\Md{0,5}\Md{2,0} - \Md{0,5}\M{4,0})u = (2k+b+c+1)(2n+a+b+c+d+2)u.
$$
\end{proof}

\subsection{Multiplication}
Recurrence relations for multiplication allow one to express $x\PPabc$, $y\PPabc$, and $z\PPabc$ in terms of a sum of Koornwinder polynomials with potentially different parameters. The recurrences in Corollary~\ref{cor:mult} are analogous to the recurrence relations for $\smash{P_n^{(a,b)}}$ found in~\cite[18.9.6]{DLMF}. 
\begin{corollary}
The following recurrence relations hold:
\begin{align}
	(2n+a+b+c+2)  x \PPabc = (n-k+a) {P}_{n,k}^{(a-1,b,c)} +(n-k+1) {P}_{n+1,k}^{(a-1,b,c)}, \label{Mx}
	\end{align}
\begin{align}
 (2k+b+c+1)(2n+a+b+c+2) y \PPabc &= (k+b)(n+k+b+c+1){P}_{n,k}^{(a,b-1,c)}\cr
 &-(k+1)(n-k+a){P}_{n,k+1}^{(a,b-1,c)}\cr
&-(k+b)(n-k+1){P}_{n+1,k}^{(a,b-1,c)}\cr
&+ (k+1)(n+k+a+b+c+2){P}_{n+1,k+1}^{(a,b-1,c)}, 
\label{My}
\end{align}
\begin{align}
	(2 k + b + c+1) (2n + a + b + c + 2)z \PPabc  &= (k + c)(n + k + b + c + 1){P}_{n,k}^{(a,b,c-1)} \cr
	&+  (k + 1)(n - k + a) {P}_{n,k+1}^{(a,b,c-1)}\cr
	& -  (k + c) (n - k + 1){P}_{n+1,k}^{(a,b,c-1)}  \cr
	 &- (k + 1)(n + k + a + b + c + 2){P}_{n+1,k+1}^{(a,b,c-1)}. 
	 \label{Mxy}
\end{align}
\label{cor:mult}
\end{corollary}
\begin{proof}
The recurrence relation in~\eqref{Mx} follows from the fact that $(\Md{3,0} + \Md{5,0})u = (2n+a+b+c+d+2) x u$. Since
$$
	(\Md{03}\Md{20} - \Md{03}\M{40} +\M{05}\M{20} - \M{05}\Md{40})u = (2k+b+c+1)(2n+a+b+c+d+2) yu
$$
holds, we find that~\eqref{My} is satisfied. Finally,~\eqref{Mxy} follows from
$$
(\Md{02}\Md{20} - \Md{02}\M{40} +\M{04}\Md{40} - \M{04}\M{20})u = (2k+b+c+1)(2n+a+b+c+d+2) zu.
$$
\end{proof}

Combining the recurrence relations in Corollary~\ref{cor:conversion} and Corollary~\ref{cor:mult}, we can derive expressions for $x\PPabc$, $y\PPabc$, and $z\PPabc$ in terms of a sum of Koornwinder polynomials with parameters $(a,b,c)$. These are analogous to the three-term recurrence relation for Jacobi polynomials~\cite[18.9.2]{DLMF}. Since these recurrence relations are long, we refer the reader to~\cite[pp.~80--81]{Dunkl_14_01}.

\subsection{Differential eigenvalue problems}
The polynomials $\PPabc$ are eigenfunctions for second-order differential operators (see~\cite[(5.3.4)]{Dunkl_14_01} and~\cite[Prop.~4.11]{Xu_TA}), and the ladder operators in Section~\ref{sec:ladder} make it easy to derive this fact.  
\begin{theorem}
The polynomial $\PPabc$ satisfies two second-order differential eigenproblems: 
$$
zy \ddyy\PPabc +((1+b)(1-x)-(2+b+c)y) \ddy\PPabc = -k (k+b+c+1) \PPabc
$$
and
$$
\begin{aligned} 
x(1-x) \ddxx\PPabc -2 x y &\ddxy\PPabc + y(1-y) \ddyy\PPabc +(a+1-(a+b+c+3) x) \ddx\PPabc \\
&+ (b+1-(a+b+c+3) y)\ddy\PPabc =  -n(n+a+b+c+2) \PPabc.
\end{aligned} 
$$
\end{theorem}
\begin{proof}
The first equation follows by considering $\M{0,1}\Md{0,1}$ and the second from
$$
\frac{zy \ddyy +((1+b)(1-x)-(2+b+c)y)\ddy}{1-x} + \M{3,0}\Md{3,0} + \M{4,0}\Md{4,0}.
$$
\end{proof}

\section*{Conclusion}
We introduce ladder operators for systematically deriving sparse recurrence relations for differentiation, conversion, and multiplication of Jacobi and Koornwinder polynomials.

\appendix

\section{Derivatives of Jacobi and Koornwinder polynomials}
In Section~\ref{sec:ladder}, it is useful to be able to express partial derivatives of $\PP$ as derivatives of shifted Jacobi polynomials. The following proposition is employed in the proof of Theorem~\ref{thm:ladder}.
\begin{proposition}
	The following relationships hold:
\begin{align}
	\tilde P_{n-k}^{(2k+b+c+d+1,a)}(x)(1-x)^k\left[\tilde  P_k^{(c,b)}\right]'\!\left(\tfrac{y}{1-x}\right) & = (1-x) \ddy \P, \label{JJp}\\
	\left[\tilde P_{n-k}^{(2k+b+c+d+1,a)}\right]'\!\!(x)(1-x)^{k+1} \tilde P_k^{(c,b)}\!\left(\tfrac{y}{1-x}\right) & = \left(k + (1 - x) \ddx - 
 y  \ddy\right) \!\P.  \label{JpJ}
\end{align}
\end{proposition}
\begin{proof}
The first relationship is immediate. The second relationship follows from the chain-rule:
$$
(1-x) \ddx \!\!\left(\! f(x) (1-x)^k g\!\left(\!\tfrac{y}{1 - x}\right)\!\right) \!=\! f'(x) (1-x)^{k+1} g\!\left(\tfrac{y}{1 - x}\right) -k f(x) (1-x)^k g(x) + y f(x) (1-x)^{k-1} g'\!\left(\tfrac{y}{1 - x}\right)
$$
and an application of~\eqref{JJp} to simplify the last term.  
\end{proof}

\end{document}